\documentclass[12pt]{amsart}
\usepackage{amsmath,amssymb}    
\usepackage{url}                
\usepackage{graphicx}           
\usepackage{amsthm}
\usepackage{amscd}
\usepackage{amsfonts}
\usepackage{latexsym}

\theoremstyle{definition}
\newtheorem{theorem}{Theorem}[section]
\newtheorem{lemma}[theorem]{Lemma}
\newtheorem{proposition}[theorem]{Proposition}
\newtheorem{corollary}[theorem]{Corollary}
\newtheorem{problem}[theorem]{Problem}
\newtheorem{definition}[theorem]{Definition}
\newtheorem{example}[theorem]{Example}

\newtheorem{conjecture}[theorem]{Conjecture}
\newtheorem{remark}[theorem]{Remark}

\newcommand{\C}{\mathbb{C}}

\newcommand{\Hom}{\text{Hom}}

\newcommand{\g}{\mathfrak{g}}
\newcommand{\ag}{\mathfrak{a}}

\date{}
\newcommand{\h}{{\mathfrak{h}}}

\begin{document}

\pagestyle{plain}
\title{Symplectic reflection algebras and affine Lie algebras}
\author{Pavel Etingof}
\address{Department of Mathematics, Massachusetts Institute of Technology,
Cambridge, MA 02139, USA}
\email{etingof@math.mit.edu}
\maketitle

\section{Introduction}

The goal of this paper is to present some results and (more
importantly) state a number of conjectures suggesting 
that the representation theory of symplectic reflection
algebras for wreath products categorifies certain structures
in the representation theory of affine Lie algebras (namely, 
decompositions of the restriction of the basic representation
to finite dimensional and affine subalgebras). 
These conjectures arose from the insight due to R. Bezrukavnikov and
A. Okounkov on the link between quantum connections for Hilbert
schemes of resolutions of Kleinian singularities and
representations of symplectic reflection algebras, and 
took a much more definite shape after my conversations with
I. Losev.  

The paper is based on my talk at the conference ``Double affine
Hecke algebras and algebraic geometry'' (MIT, May 18, 2010). 
I'd like to note that it is not a complete study of the subject, but
rather an attempt to give an outline for further investigation; 
at many places, it has speculative nature, and the picture it 
presents is largely conjectural. 

The plan of the paper is as follows. In Section 2, 
we review the preliminaries on the main objects of study - 
symplectic reflection algebras for wreath product groups 
$\Gamma_n=S_n\ltimes \Gamma^n$, $\Gamma\subset SL_2(\Bbb C)$, 
and affine Lie algebras. In Section 3, we define a filtration 
$F_\bullet$ on the Grothendieck group of
the category of representations of $\Gamma_n$ coming from
symplectic reflection algebras. If $\Gamma$ is cyclic, we define
another filtration $\bold F_\bullet$ on the same group, and 
conjecture that they have the same Poincar\'e polynomials. 
In Section 4, we use Ext groups to define the inner product on the Grothendieck
group of the category of finite dimensional representations of 
the symplectic reflection algebra, as well as its
$q$-deformation, and conjecture that the inner product is
positive definite, and that the q-deformation degenerates
at roots of unity. In Section 5, we present results and
conjectures on singular and aspherical hyperplanes 
for symplectic reflection algebras for wreath products, and show
that the aspherical locus coincides with the locus of infinite
homological dimension for the spherical subalgebra. Finally, in
Section 6 we present the main conjectures, which interpret
the homogeneous components of the associated graded spaces for the above
filtrations in terms of the decomposition of the basic
representation of the affine Lie algebra corresponding to
$\Gamma$ via the McKay corresondence (tensored with one copy of
the Fock space) to a finite dimensional or affine subalgebra. 
In this section, we also explain the connection of our
conjectures with the work of Bezrukavnikov and Okounkov, with the
work of Gordon-Martino and Shan, and give motivation for the
conjectures on the inner products (they follow from the
conjectural interpretation of the Ext inner product as the
Shapovalov form on the basic representation). We also 
present evidence for our conjectures, by 
discussing a number of cases when they can be deduced
from the results available in the literature. 

{\bf Acknowledgements.} I would never have thought of these things 
without the encouragement and the vision of R. Bezrukavnikov and A. Okounkov.  
Also, I am very grateful to R. Bezrukavnikov and I. Losev 
for numerous discussions, without which I would have gotten nowhere. 
This work was supported by the NSF grant DMS-0854764.

\section{Symplectic reflection algebras, Gan-Ginzburg
algebras, and affine Lie algebras} 

\subsection{Symplectic reflection algebras for wreath
products}
Let $\Gamma\subset SL_2(\Bbb C)$ be a finite subgroup. 
For $n\ge 1$, let $\Gamma_n$ be the wreath product $S_n\ltimes
\Gamma^n$. Let $V=\Bbb C^2$ be the tautological representation of
$\Gamma$, and $(,)$ be a symplectic form on
$V$. Let $V_n=V\otimes \Bbb C^n$ the corresponding
representation of $\Gamma_n$. Note that the form $(,)$ gives rise to a
symplectic structure on $V_n$ invariant under $\Gamma_n$. 

Let $k\in \Bbb C$, and 
$c: \Gamma\to \Bbb C$ be a conjugation-invariant function. 
Thus, $c=(c_0,...,c_r)$, where $c_i$ is the value of $c$ on the
$i$-th conjugacy class $C_i$ of $\Gamma$ (where $C_0=\lbrace{1\rbrace}$).

For $v\in V$, let $v_i=(0,...,v,...,0)\in V_n$ (where $v$ stands
in the $i$-th place), and similarly for $\gamma\in \Gamma$ let 
$\gamma_i=(1,...,\gamma,...,1)\in \Gamma_n$. Let $s_{ij}=(ij)\in
S_n$. 

\begin{definition} (\cite{EG})
The symplectic reflection algebra $H_{c,k}(\Gamma_n)$ 
is the quotient of $\Bbb C\Gamma_n\ltimes TV_n$ by the relations 
$$
[u_i,v_j]=k\sum_{\gamma\in \Gamma} (\gamma
u,v)s_{ij}\gamma_i\gamma_j^{-1},\ i\ne j;
$$
$$
[u_i,v_i]=(u,v)(\sum_{\gamma\in \Gamma} c_\gamma\gamma_i-
k\sum_{j\ne i}\sum_{\gamma\in \Gamma}
s_{ij}\gamma_i\gamma_j^{-1}),
$$
for all $u,v\in V$. 
\end{definition}

\begin{remark} 1. $H_{0,0}(\Gamma_n)=\Bbb C\Gamma_n\ltimes SV_n$.

2. If $\Gamma\ne 1$, $H_{c,k}(\Gamma_n)$ is the universal 
filtered deformation of $H_{0,0}(\Gamma_n)$. 

3. For $n=1$, the algebra $H_{c,k}(\Gamma_n)$ is independent of $k$; 
we will denote it by $H_c(\Gamma)$. 

4. One has $H_{c,0}(\Gamma_n)=\Bbb CS_n\ltimes
H_c(\Gamma)^{\otimes n}$. 
\end{remark}

\begin{example}
Let $\Gamma=1$, and $c_0=t$. Then we get the algebra 
$H_{t,k}(S_n)$ which is the quotient of $\Bbb CS_n\ltimes TV_n$
by the relations
$$
[x_i,x_j]=0, [y_i,y_j]=0,
$$
$$
[y_i,x_j]=ks_{ij},\ [y_i,x_i]=t-k\sum_{j\ne i}s_{ij}, 
$$
where $\lbrace x,y\rbrace$ is the standard basis of $V=\Bbb C^2$
(such that $(y,x)=1$). This is the rational Cherednik algebra for
$S_n$ with parameters $t,k$. 
\end{example}

Below we will need the following proposition. 

\begin{proposition}\label{clo} (I. Losev) Let $E$ be a finite dimensional
representation of $\Gamma_n$. Then the set $Z_E$ of values of $c,k$
with $c_0=1$ such that $E$ extends to a representation of 
$H_{c,k}(\Gamma_n)$ is Zariski closed in $\Bbb A^{r+1}$.  
\end{proposition}

\begin{proof}
Let $Y=\Hom_{\Gamma_n}(V_n,{\rm End}(E))$, and let 
$\widehat{Z}_E\subset Y\times \Bbb A^{r+2}$ 
be the set of maps that give a representation 
of $H_{c,k}(\Gamma_n)$ (without the assumption 
that $c_0=1$). Let $\widetilde{Z}_E$ be the closure of 
the intersection of $\widehat{Z}_E$ with the set $c_0\ne 0$. 
We have an action of the reductive group ${\rm Aut}_{\Gamma_n}(E)$
on $\widetilde{Z}_E$, and a map 
$$
\phi: \widetilde{Z}_E/{\rm Aut}_{\Gamma_n}(E)\to \Bbb A^{r+2}.
$$
from the categorical quotient to the space of parameters. 
The set $Z_E$ is the intersection of ${\rm Im}\phi$ with the
hyperplane $c_0=1$, so it suffices to show that ${\rm Im}\phi$ is
closed. This follows from the following lemma. 

\begin{lemma} $\phi$ is a finite morphism. 
\end{lemma}

\begin{proof}
Let $X$ be an affine variety with a contracting $\Bbb
C^*$-action. It is well known that if 
$f_1,...,f_n$ are homogeneous regular functions 
on $X$ of positive degrees such that the equations
$f_1=0,...,f_n=0$ cut out the point $0\in X$
(set-theoretically), then the map $X\to \Bbb A^n$ induced by
$f_1,...,f_n$ is finite. 

We apply this to $X=\widetilde{Z}_E/{\rm Aut}_{\Gamma_n}(E)$
and $f_i$ being the coordinates $c_j,k$ on $\Bbb A^{r+2}$. Our job is
just to check that the zero set $X_0$ 
of the equations $c_j=0$, $k=0$ 
on $X$ is just a single point, 
i.e. the representation $E$ of $\Gamma_n$ on which $V_n$ acts by zero. 

To see this, note that $X_0$ consists of semisimplifications of 
representations $W$ of $H_{0,0}(\Gamma_n)$ which are 
isomorphic to $E$ as $\Gamma_n$-representations and which 
are obtained as degenerations of representations of 
$H_{c,k}(\Gamma_n)$ with $c_0\ne 0$. 

We claim that such a representation $W$ 
must be supported at $0\in V_n$ as an $SV_n$-module. Indeed, 
let $\psi: H_{c(\hbar),k(\hbar)}(\Gamma_n)\to {\rm
End}(W)[[\hbar]]$ be a formal 1-parameter deformation 
of $W$, with $k(0)=0$, $c(0)=0$, but $c_0(\hbar)\ne 0$. 
Since $W$ lies in the closure of $\widehat{Z}_E\cap
\lbrace{c_0\ne 0}$, such a deformation must exist. 
Let $B={\rm Im}\psi$, and $B_0=B/\hbar B$. 
Since $B$ is a $\Bbb C[[\hbar]]$-free quotient algebra of $H_{c(\hbar),k(\hbar)}(\Gamma_n)$, 
$B_0$ is a finite dimensional Poisson module over the center $(SV_n)^{\Gamma_n}$
of $H_{0,0}(\Gamma_n)$. Hence, $B_0$ is supported at the zero-dimensional symplectic leaves
of $V_n/\Gamma_n$ in the Poisson structure induced by the symplectic structure on $V_n$. 
But the only such symplectic leaf is $0\in V_n/\Gamma_n$. 
Thus, $B_0$ is supported at $0$ as an $SV_n$-module.
Now, the image $B_0'$ of $H_{0,0}(\Gamma_n)$ in ${\rm End}(W)$ 
is a quotient of $B_0$, so it is also supported at $0$. Hence, so is $W$. 
 
This implies that in the semisimplification of $W$, the space $V_n$ must act by zero, which proves
the lemma.    
\end{proof} 

The proposition is proved. 

\end{proof} 

\begin{remark} It is clear that Proposition
\ref{clo} holds, with the same proof, 
for symplectic reflection algebra 
attached to any finite subgroup of $Sp(2n)$. 
\end{remark}

\subsection{The Gan-Ginzburg algebras} 

For any quiver $Q$, Gan and Ginzburg defined algebras 
$A_{n,\lambda,k}(Q)$, parametrized by $n\in \Bbb N$, 
$k\in \Bbb C$, and a complex function $\lambda$ 
on the set $I$ of vertices of $Q$ (\cite{GG}). 
We refer the reader to \cite{GG} for the precise definition;
let us just note that if $n=1$ then this algebra does not depend
on $k$, and is the deformed preprojective algebra 
$\Pi_\lambda(Q)$ defined by Crawley-Boevey and Holland
(\cite{CBH}), and that $A_{n,\lambda,0}(Q)=\Bbb CS_n\ltimes
\Pi_\lambda(Q)^{\otimes n}$. 

It turns out that if $Q=Q_\Gamma$ is the affine quiver of ADE type 
corresponding to $\Gamma$ via McKay's correspondence, then 
the algebra $A_{n,\lambda,k}(Q)$ is Morita equivalent to 
$H_{c,k}(\Gamma_n)$ under a certain correspondence between
$\lambda$ and $c$. Namely, recall that McKay's correspondence 
provides a labeling of the irreducible characters $\chi_i$ of
$\Gamma$ by vertices $i\in I$ of $Q$. 

\begin{theorem}\label{moreq} (\cite{GG}) 
If 
\begin{equation}\label{lc}
\lambda(i)=\frac{1}{|\Gamma|}\sum_{\gamma\in
\Gamma}\chi_i(\gamma)c_\gamma
\end{equation}
then the Gan-Ginzburg 
algebra $A_{n,\lambda,k}(Q)$ is naturally Morita equivalent 
to $H_{c,k}(\Gamma_n)$. 
\end{theorem} 

\begin{remark}
1. The inverse transformation to (\ref{lc}) is 
$$
c_\gamma=\sum_{i\in I}\chi_i(\gamma^{-1})\lambda(i).
$$
In particular, $c_0=\sum_i \chi_i(1)\lambda(i)$. 

2. If $\Gamma$ is cyclic, the Morita equivalence of Theorem \ref{moreq} 
is actually an isomorphism. 
\end{remark}

In these notes, we will be interested in representation theory
questions for the algebra $H_{c,k}(\Gamma_n)$. These questions
will always be Morita invariant, so they will be equivalent to the
same questions about the Gan-Ginzburg algebra
$A_{n,\lambda,k}(Q)$; i.e., it does not matter which algebra to
use. 

\subsection{Affine Lie algebras}

It turns out that $\lambda$ is a more convenient
parameter than $c$. Namely, $\lambda$ may be interpreted in terms of 
affine Lie algebras. Before explaining this interpretation, 
let us review the basics on affine Lie algebras (cf. \cite{K}). 

Let $\g$ be the finite dimensional simply laced simple
Lie algebra corresponding to the affine Dynkin diagram $Q_\Gamma$
(agreeing that $\g=0$ for $\Gamma=1$), let $\h$ be its Cartan
subalgebra, let $\widehat{\g}=\g [t,t^{-1}]\oplus \Bbb CK$ 
be the affinization of $\g$, and
let $\widehat{\h}=\h\oplus \Bbb CK$ be the Cartan subalgebra of the affine 
Lie algebra $\widehat{\g}$. Define the extended affine Lie algebra 
$\widetilde{\g}=\widehat{\g}\oplus \Bbb CD$, 
where $[D,b(t)]=tb'(t)$. Its Cartan subalgebra is 
$\widetilde \h=\widehat{\h}\oplus \Bbb CD$.
The space $\widetilde{\h}$ carries a nondegenerate symmetric bilinear form 
obtained by extending the form on $\widehat{\h}$ via $(D, D)=0$, 
$(D, \h)=0, (D, K)=1$. This form defines a form on the dual space $\widetilde{\h}^*$.

Let $\omega_i\in \widetilde{\h}^*$ be the fundamental weights of $\widetilde{\g}$
(so we have $\omega_i(D)=0$), let $\alpha_i\in \widetilde{\h}^*$ be its simple positive roots, and 
let $\delta=\sum \chi_i(1)\alpha_i$ be the basic imaginary root. 
Then $\omega_i$ and $\delta$ form a basis of $\widetilde{\h}^*$. 

Now we can interpret $\lambda$ as a 
weight for $\widetilde{\g}$, i.e. $\lambda\in \widetilde{\h}^*$. 
Namely, 
$$
\lambda=\sum_{i\in I}\lambda(i)\omega_i
$$
(so $\lambda(i)=(\lambda,\alpha_i)$).

We will be interested in the ``quantum'' case $c_0\ne 0$,
i.e. when the center of our algebras is trivial. 
Since the parameters can be simultaneously renormalized, we may
assume that $c_0=1$. In terms of $\lambda$, this condition is
written as $(\lambda,\delta)=1$.

\section{Filtrations on $K_0(H_{c,k}(\Gamma_n))$}

Now fix the parameters $c,k$ and consider the group
$K_0(H_{c,k}(\Gamma_n))$ (formed by finite projective modules
modulo stable equivalence). Since ${\rm gr}(H_{c,k}(\Gamma_n))=
H_{0,0}(\Gamma_n)=\Bbb C\Gamma_n\ltimes SV_n$, by 
a standard theorem in algebraic K-theory we have 
a natural isomorphism 
$$
\psi: K_0({\rm Rep} \Gamma_n)\to K_0(H_{c,k}(\Gamma_n)),
$$
where for a finite group $G$, ${\rm Rep} G$ denotes 
the category of finite dimensional complex
representations of $G$ (this isomorphism sends 
$\tau\in {\rm Rep} \Gamma_n$ to the projective $H_{c,k}(\Gamma_n)$-module 
$\psi(\tau):={\rm Ind}_{\Bbb C\Gamma_n}^{H_{c,k}(\Gamma_n)}\tau$). 
Also, we see that $H_{c,k}(\Gamma_n)$ has finite homological
dimension, which implies that every finitely generated
$H_{c,k}(\Gamma_n)$-module $M$ gives a class $[M]\in  
K_0(H_{c,k}(\Gamma_n))$.

Now we would like to define an increasing filtration 
on $K_0(H_{c,k}(\Gamma_n))$: $F_0\subset F_1\subset...\subset
F_n=K_0$. 

To do so, for any finitely generated
$H_{c,k}(\Gamma_n)$-module $M$ let ${\rm Ann}(M)$ denote 
its annihilator, and consider the corresponding graded ideal 
${\rm gr}({\rm Ann}(M))\subset \Bbb C\Gamma_n\ltimes SV_n$. 
Define the {\it annihilator variety} $AV(M)$ to be the zero set of the intersection 
${\rm gr}({\rm Ann}(M))\cap SV_n$ in $V_n^*$. This is a
$\Gamma_n$-invariant subset, and since 
${\rm gr}({\rm Ann}(M))\cap (SV_n)^{\Gamma_n}$
is a Poisson ideal, it is a union of strata of the (finite) stratification 
of $V_n$ by stabilizers of points in $\Gamma_n$. 
Note that all these strata are symplectic and hence even
dimensional. 

Now for $0\le i\le n$ define $F_iK_0(H_{c,k}(\Gamma_n))$ to be the span of classes
$[M]$ of modules $M$ such that $AV(M)$ has dimension at most
$2i$. By transport of structure using $\psi$, this gives us a
filtration on $K_0({\rm Rep} \Gamma_n)$, which we denote by
$F_\bullet^{c,k}$. 

\begin{example}\label{rk1}
Let $n=1$. Then the only nontrivial piece of the filtation $F_\bullet$ is $F_0$, which is the span of the classes 
of representations with zero dimensional annihilator variety (i.e., finite dimensional representations).
We may assume that $\Gamma\ne 1$, since otherwise there is no finite dimensional 
representations and the filtration $F_\bullet$ is trivial. Let $r\ge 1$ be the number of nontrivial conjugacy classes of $\Gamma$.
Let $\lambda=\omega_0$. In this case, simple finite dimensional modules $L_i$ are the irreducible nontrivial 
$\Gamma$-modules $\chi_i$, $i=1,...,r$, with the zero action of $V$ (see \cite{CBH}). 
Let $e_j$ be primitive idempotents of the representations $\chi_j$, 
${\mathcal P}_j=H_c(\Gamma)e_j$ be the corresponding projective modules.
We have ${\rm gr}{\mathcal P}_j=SV\otimes \chi_j$. Consider the Koszul resolution 
of the $\Bbb C\Gamma\ltimes SV$-module $\chi_j$ with the trivial action of $V$: 
$$
0\to SV\otimes \chi_j\to SV\otimes V\otimes \chi_j\to SV\otimes \chi_j\to \chi_j\to 0.
$$
Since 
by McKay's correspondence 
$V\otimes \chi_j=\oplus_{i-j}\chi_i$ (where $i-j$ means that $i$
is connected to $j$ in the quiver) this resolution can be written as 
$$
0\to SV\otimes \chi_j\to \oplus_{i-j}SV\otimes\chi_i\to SV\otimes \chi_j\to \chi_j\to 0.
$$
Lifting this resolution to the filtered situation, we get the resolution 
$$
0\to {\mathcal P}_j\to \oplus_{j-i}{\mathcal P}_i\to {\mathcal P}_j\to L_j\to 0.  
$$
This means that in the Grothendieck group, we have
\begin{equation}\label{resol}
L_i=\sum a_{ij}{\mathcal P}_j,
\end{equation}
and $A=(a_{ij})$ is the Cartan matrix of the quiver $Q_\Gamma$. 

Thus, we see that $F_0$ is spanned by $\sum_{j\in I}a_{ij}{\mathcal P}_j$ 
for all $i\in I$, $i\ne 0$. This implies that $F_1/F_0=\Bbb Z\oplus C_\Gamma$, 
where $C_\Gamma$ is the center of the corresponding simply connected simple Lie group $G$.
\end{example}

If $\Gamma$ is a cyclic group, then we can define another filtration 
on $K_0({\rm Rep} \Gamma_n)$, which we will denote by ${\bf F}_\bullet$. 
Namely, in this case $H_{c,k}(\Gamma_n)$ is a rational Cherednik algebra, 
and we can define the category ${\mathcal O}_{c,k}(\Gamma_n)$ of 
finitely generated modules over this algebra which are locally nilpotent 
under the action of the polynomial subalgebra $\Bbb C[y_1,...,y_n]\subset H_{c,k}(\Gamma_n)$,
where $x,y$ is a symplectic basis of $V$ which is also an eigenbasis for $\Gamma$.  
We have an isomorphism $\eta: K_0({\rm Rep} \Gamma_n)\to K_0({\mathcal O}_{c,k}(\Gamma_n))$ 
which sends $\tau\in {\rm Rep} \Gamma_n$ to the indecomposable projective object $P_{c,k}(\tau)$ covering the 
simple module $L_{c,k}(\tau)$ with lowest weight $\tau$. Then we can define 
${\bf F}_iK_0({\mathcal O}_{c,k}(\Gamma_n))$ to be the span of the classes of modules whose support 
as $\Bbb C[x_1,...,x_n]$-modules has dimension at most $i$, and define ${\bold F}_i^{c,k}K_0({\rm Rep} \Gamma_n)$ 
via transport of structure by $\eta$. 

\begin{problem} Describe the filtrations $F_\bullet^{c,k}$ and ${\bold F}_\bullet^{c,k}$. In
particular, find their Poincar\'e polynomials as functions of $c,k$. 
\end{problem}

\begin{conjecture}
The filtrations ${\bf F}_\bullet^{c,k}$ and $F_\bullet^{c,k}$ have the same Poincar\'e polynomial. 
\end{conjecture} 

\begin{example}
Let $\Gamma=\Bbb Z_\ell$ (the cyclic group of order $\ell>1$), 
and $n=1$. Take $\lambda=\omega_0$ (so $(\lambda,\alpha_i)=0$ for all $i>0$). 
In this case, denote Verma modules in category $\mathcal{O}$ by $M_1,...,M_\ell$ (they are 
labeled by the characters of $\Gamma$). We have inclusions $M_\ell\subset...\subset M_1$, 
and the simple modules are the 1-dimensional modules $L_i=M_i/M_{i+1}$ for $i<\ell$, and 
$L_\ell=M_\ell$. Thus, in the Grothendieck group of category ${\mathcal O}$ we have
$$
M_i=\sum_{j\ge i}L_j.
$$
Hence, by BGG duality for the projective objects we have 
$$
P_k=\sum_{i\le k}M_i.
$$
So $M_k=P_k-P_{k-1}$ for $k>1$, and $M_1=P_1$. Hence for $1<i<\ell$, 
$L_i=2P_i-P_{i-1}-P_{i+1}$, while $L_1=2P_1-P_2$.  
Thus we see that if $\chi_1,\chi_2,...,\chi_\ell$ is the 
standard basis of ${\rm Rep} \Gamma$ (where $\chi_\ell$ is the trivial character), 
then ${\bold F}_0K_0({\rm Rep} \Gamma)$ 
is the subgroup of rank $\ell-1$ spanned by 
$2\chi_i-\chi_{i-1}-\chi_{i+1}$ for 
$i<\ell$ (where $\chi_0$ should be dropped). 
 
On the other hand, as explained in Example \ref{rk1}, we have  
$$
L_i=2{\mathcal P}_i-{\mathcal P}_{i-1}-{\mathcal P}_{i+1},
$$
where the indexing is understood cyclically. 
So $F_0K_0({\rm Rep} \Gamma)$ 
is the subgroup of rank $\ell-1$ spanned by 
$2\chi_i-\chi_{i-1}-\chi_{i+1}$ for 
$i<\ell$ (where $\chi_0$ should be interpreted as $\chi_\ell$). 

We see that while the Poincar\'e polynomials of both filtrations
are the same (and equal $1+(\ell-1)t$), the filtrations are not quite the same 
and even are not isomorphic (although they, of course, become isomorphic after tensoring with $\Bbb Q$). 
Indeed, the quotient group for the filtration $\bold F_\bullet$ is free (i.e., $\Bbb Z$),
while for the filtration $F_\bullet$ it is not free (namely, it is 
$\Bbb Z\oplus \Bbb Z_\ell$).  

In a similar manner one checks the coincidence of the Poincar\'e
polynomials for any value of $\lambda$ (or $c$) and $n=1$;
namely, both polynomials are equal to $\ell+m(t-1)$, where $m$ is the
dimension of the span of the roots $\alpha$ of $\g$ such that $(\lambda,\alpha)$ 
is an integer. This implies that the Poincar\'e polynomials also 
coincide for $H_{c,0}(\Gamma_n)$ for $n>1$.  
\end{example}

\section{The inner product on the Grothendieck group of 
the category of finite dimensional modules} 

\subsection{The inner product $(,)$}
It is known (see e.g. \cite{ES}) that $H_{c,k}(\Gamma_n)$ 
has finitely many irreducible finite dimensional modules. 
So we can define the finitely generated free abelian group $K_0(H_{c,k}(\Gamma_n)-{\rm mod}_f)$ 
(where the subscript $f$ denotes finite dimensional modules). 
We have a natural pairing 
$$
B: K_0(H_{c,k}(\Gamma_n)-{\rm mod})\times K_0(H_{c,k}(\Gamma_n)-{\rm mod}_f)\to \Bbb Z
$$
 given by 
$$
B(M,N)=\dim {\rm R}\Hom(M,N),
$$
where the dimension is taken in the supersense. 
Note that this makes sense since $H_{c,k}$ has finite homological dimension. 
Also, the same formula, $(M,N)=\dim {\rm R}\Hom(M,N)$, 
defines an inner product on $K_0(H_{c,k}(\Gamma_n)-{\rm mod}_f)$. 

\begin{conjecture}\label{pd}
The inner product (,) is symmetric and positive definite 
(in particular, nondegenerate). 
\end{conjecture}

The motivation for Conjecture \ref{pd} is explained in Subsection \ref{pdmot}. 

Conjecture \ref{pd} would imply 

\begin{conjecture}\label{pd1} 
The natural map 
$$
\zeta: K_0(H_{c,k}(\Gamma_n)-{\rm mod}_f)\to F_0K_0(H_{c,k}(\Gamma_n)-{\rm mod})
$$ 
is injective 
(hence an isomorphism). 
\end{conjecture} 

Indeed, $B(\zeta(M),N)=(M,N)$. 

\begin{example}
Let $n=1$. In this case, from (\ref{resol}) 
we see that $(L_i,L_j)=a_{ij}$, and Conjecture \ref{pd}
follows. 

In a similar way, using the results of \cite{CBH}, one can show that Conjecture \ref{pd} holds for 
any $\lambda$ and $n=1$. This implies that it holds for $H_{c,0}(\Gamma_n)$ for $n>1$. 
\end{example}
 
\begin{proposition}\label{equa}\footnote{The author is grateful
to I. Losev, who caught an error in the original proof of this proposition.}
Let $\Gamma$ be a cyclic group. 
Then for any $H_{c,k}(\Gamma_n)$-modules 
$M,N$ from category ${\mathcal O}$, one has 
$$
{\rm Ext}^i_{\mathcal O}(M,N)\cong {\rm Ext}^i(M,N), 
$$
where the subscript ${\mathcal O}$ means that Ext 
is taken in category ${\mathcal O}$.  
\end{proposition}

\begin{proof}
Fix a resolution of $M$ by finitely generated free 
modules over $H:=H_{c,k}(\Gamma_n)$:
$$
...\to F_1\to F_0\to M\to 0.
$$
Then ${\rm Ext}^\bullet(M,N)$ is the cohomology of the 
corresponding complex 
\begin{equation}\label{comp}
\Hom_H(F_0,N)\to \Hom_H(F_1,N)\to...
\end{equation}
Let $\widehat{H}$ be the completion of $H$ 
near $0$ as a right $\Bbb C[y_1,...,y_n]$-module:
$$
\widehat{H}:=H\hat\otimes_{\Bbb C[y_1,...,y_n]}\Bbb C[[y_1,...,y_n]]:=\underleftarrow{\lim}H/H{\mathfrak{m}}^r,
$$ 
where ${\mathfrak{m}}$ is the augmentation ideal in $\Bbb C[y_1,...,y_n]$
Then the $H$-action on $M$ and $N$ extends to
a continuous action of the completed algebra 
$\widehat{H}$. Set $\widehat{F_i}=\widehat{H}\otimes_H F_i$.    
Then complex (\ref{comp}) coincides with the complex
\begin{equation}\label{comp1}
\Hom_{\widehat{H}}(\widehat{F_0},N)\to \Hom_{\widehat{H}}(\widehat{F_1},N)\to...,
\end{equation}

\begin{lemma}\label{reso}
The sequence 
$$
...\to\widehat{F_1}\to \widehat{F_0}\to M\to 0
$$
is exact, i.e., it is a resolution of $M$ as an $\widehat{H}$-module.
\end{lemma}

\begin{proof}
{\it Step 1.} {\it One has ${\widehat
H}\otimes_HM=M$.} Indeed, we have natural maps 
$M\to {\widehat H}\otimes_HM\to M$ (the second map 
is the action of $\widehat{H}$ on $M$), the composition is the
identity, and the first map is surjective\footnote{Indeed, 
if $\sum_{i\ge 0}a_i$, $a_i\in H$, is a convergent series 
in $\widehat{H}$ then for any $v\in M$, there is $i_0$ 
such that for $i\ge i_0$ one has $a_iv=0$, so 
$(\sum_{i\ge 0}a_i)\otimes v=\sum_{i=0}^{i_0}a_i\otimes v$, which 
is the image of $\sum_{i=0}^{i_0}a_iv\in M$.}. 

Thus, it suffices to show that 
${\rm Tor}_i^H(\widehat{H},M)=0$ for $i>0$. 

{\it Step 2.}  
{\it One has ${\rm Tor}_i^H(\widehat{H},M)=0$ for $i>0$
if $M=M(\tau)$ is a Verma module.} 

Indeed, 
in this case we can take $F_\bullet$ to be the Koszul resolution:
$F_j=H\otimes \wedge^j(y_1,...,y_n)\otimes \tau$, with the differential 
being the Koszul differential written in terms of right multiplication by $y_i$. Then $\widehat{F}_j
=\widehat{H}\otimes \wedge^j(y_1,...,y_n)\otimes \tau$ with differential defined in the same way, 
so the exactness of $\widehat{F}_\bullet$ follows from the following claim.  

{\bf Claim.} Let $E$ be any vector space, and let $C_\bullet=E[[y_1,...,y_n]]\otimes \wedge^\bullet(y_1,...,y_n)$, equipped with the 
Koszul differential. Then $C_\bullet$ is exact in all nonzero degrees, and its zeroth homology is $E$. 

{\it Step 3.} {\it One has ${\rm Tor}_i^H(\widehat{H},M)=0$ for $i>0$, for any $M\in {\mathcal O}$.} 

This is shown by induction in $i$. For brevity, write $T_i(M)$ for 
${\rm Tor}_i^H(\widehat{H},M)$. For each particular $i$, by the long exact sequence of homology, 
it suffices to prove that $T_i(M)=0$ for irreducible $M$ (as any module in ${\mathcal O}$ has finite length). 

The base of induction ($i=1$) also follows from the long exact sequence of homology. Indeed, 
let $M=L(\tau)$ be irreducible, and $0\to K\to M(\tau)\to M\to 0$ be a short exact sequence. By Steps 1 and 2, 
a portion of the corresponding long exact sequence looks like: 
$$
0=T_1(M(\tau))\to T_1(M)\to K\to M(\tau)\to M\to 0,
$$
and the map $K\to M(\tau)$ is the same as in the short exact sequence, so $T_1(M)=0$.

It remains to fulfill the induction step. Assume that the statement is known for $i=m$ and let us prove it for $i=m+1$. 
Again assuming $M$ is irreducible and considering the above short exact sequence, we get from the corresponding long exact sequence
(using Step 2 and the induction assumption): 
$$
0=T_{m+1}(M(\tau))\to T_{m+1}(M)\to T_m(K)=0,
$$
so $T_{m+1}(M)=0$, as desired. 

The lemma is proved.

\end{proof}

By Lemma \ref{reso}, 
the sequence 
$$
...\to\widehat{F_1}\to \widehat{F_0}\to M\to 0
$$
is a resolution of $M$ in the pro-completion of 
category ${\mathcal O}$. Since category ${\mathcal O}$ has enough projectives (\cite{GGOR}), 
this resolution is quasiisomorphic to a resolution inside category ${\mathcal O}$. 
Thus, complex (\ref{comp1}) computes ${\rm Ext}_{\mathcal O}^\bullet(M,N)$, as desired. 
\end{proof}

\begin{corollary}
 Conjecture \ref{pd} (and hence \ref{pd1}) holds
for the case of cyclic $\Gamma$ (and arbitrary $n,c,k$). 
\end{corollary}

\begin{proof}
Let $L_i$, 
$i=1,...,p$, be the irreducible modules in ${\mathcal O}$.
Let $P_i$ be the projective covers of $L_i$, and $M_i$ the standard modules.  
Then by the BGG reciprocity
$$
[P_i,L_j]=\sum_s [P_i:M_s][M_s:L_j]=\sum_s [M_s:L_i][M_s:L_j].
$$
So if $N$ is the matrix of multiplicities $[M_s:L_i]$ then we see 
that the Cartan matrix $C=([P_i,L_j])$ is given by the formula $C=N^TN$. 
Thus $C^{-1}$ is symmetric and positive definite. On the other hand, 
using Proposition \ref{equa}, it is clear that $(L_i,L_j)=(C^{-1})_{ij}$ if $L_i,L_j$ are finite dimensional.
So the Gram matrix of the form $(,)$ in the basis $L_i$ is a principal submatrix of the 
matrix $C^{-1}$. This implies Conjecture \ref{pd}. 
\end{proof}

\subsection{The q-deformed inner product}
One can also consider the q-deformation of the inner product $(,)$: 
$$
(M,N)_q=\sum (-q)^j\dim {\rm Ext}^j(M,N). 
$$

\begin{conjecture}\label{rootun}
If $q$ is not a root of unity, then the form $(,)_q$ is nondegenerate. 
\end{conjecture} 

The motivation for Conjecture \ref{rootun} is explained in Subsection \ref{pdmot}. 

\begin{example} 
Let $\Gamma\ne 1$, $n=1$, $\lambda=\omega_0$. 
Then by (\ref{resol}), the matrix of $(,)_q$ in the basis 
of simple modules is the q-deformed Cartan matrix 
$A_q$: $(a_q)_{ii}=1+q^2$, 
$(a_q)_{ij}=-q$ if $i$ is connected to $j$, and 
zero otherwise. 
It is well known from the work of Lusztig and Kostant that 
$\det(A_q)$ is a ratio of products of binomials of the form $1-q^l$.
This implies Conjecture \ref{rootun} in this case. 
Similarly one handles the case of general $\lambda$ and 
the case of $H_{c,0}(\Gamma_n)$, $n>1$. 
\end{example}

\begin{remark}
After this paper appeared online, Conjecture \ref{rootun} for cyclic 
groups $\Gamma$ under some restrictions on the parameters 
was proved by Gordon and Losev, \cite{GL}. 
\end{remark}

\section{The singular and aspherical hyperplanes} 

\subsection{Singular hyperplanes} 
I. Losev (\cite{Lo}, Theorem 1.4.2) showed that the algebra $H_{c,k}(\Gamma_n)$ 
is simple outside of a countable collection of hyperplanes 
in the space of parameters $c,k$. The following conjecture makes this statement more precise. 

\begin{conjecture}\label{sinhyp}
The algebra $H_{c,k}(\Gamma_n)$ is simple 
if and only if $(\lambda,k)$ does not belong to the hyperplanes
$$
E_{m,N}: km+N=0,\ m\in \Bbb Z,\ 2\le m\le n,\ N\in \Bbb Z,\ {\rm gcd}(m,N)=1, 
$$
and the hyperplanes 
$$
H_{\alpha,m,N}: (\lambda,\alpha)+km+N=0,
$$
where $m$ is an integer with $|m|\le n-1$, $\alpha$ is a root of $\g$, 
and $N\in \Bbb Z_{\ge 0}$. 
\end{conjecture}

Conjecture \ref{sinhyp} holds for $n=1$ (by \cite{CBH}), and also 
for any $n$ in the case of cyclic $\Gamma$. Indeed, in this case,  
the simplicity of $H_{c,k}$ is equivalent to the semisimplicity of the corresponding 
cyclotomic Hecke algebra (due to R. Vale; see \cite{BC}, Theorem 6.6), 
which is known to be semisimple exactly away from the above hyperplanes. 

\subsection{The aspherical locus}
It is also interesting to consider the aspherical locus. 
Namely, let $e_{\Gamma_n}=\frac{1}{|\Gamma_n|}\sum_{g\in \Gamma_n}g$ be the averaging element for $\Gamma_n$. 

\begin{definition} The aspherical locus ${\mathcal A}(\Gamma_n)$ 
is the set of $(\lambda,k)$ such that there exists a nonzero $H_{c,k}(\Gamma_n)$-module $M$ with $e_{\Gamma_n}M=0$
(i.e., equivalently, the two-sided ideal generated by $e_{\Gamma_n}$ in $H_{c,k}(\Gamma_n)$ is a proper ideal).  
\end{definition}

\begin{conjecture}\label{asph}
The aspherical locus is the union of the hyperplanes
$E_{m,N}$ for $1\le N\le m-1$ and the hyperplanes
$H_{\alpha,m,N}$ for $|m|\le n-1$ and 
$$
0\le N\le  \sqrt{n+\frac{m^2}{4}}+\frac{m}{2}-1.
$$
\end{conjecture}

For cyclic $\Gamma$, this conjecture is proved in \cite{DG}. 
It is also easy to prove it for $n=1$ (in this case it follows from \cite{CBH}). 
Also, for any $\Gamma,n$ we have 

\begin{theorem}\label{conta} The hyperplanes specified 
in Conjecture \ref{asph} are contained in the aspherical locus. 
\end{theorem}

\begin{proof}
Let us start with the hyperplanes $k=-N/m$. 
It follows from \cite{Lo}, Theorem 1.2.1, that it suffices to show that 
a point on this hyperplane gives rise to an aspherical point 
for some parabolic subgroup $W\subset \Gamma_n$. 
Take the parabolic subgroup $S_n\subset \Gamma_n$ (stabilizer of a point $(v,...,v)$, 
where $0\ne v\in V$). It is well known that $k=-N/m$ with $N,m$ as in the theorem are 
aspherical for $S_n$. So we are done in this case. 

Consider now the hyperplanes $H_{\alpha,m,N}$. 
Let us first show that the generic point of  
each hyperplane is contained in the aspherical locus. 
By \cite{Lo}, to this end, it suffices to show that for a generic point 
of each hyperplane there is a finite dimensional 
representation of $H_{c,k}(\Gamma_q)$ which 
is killed by $e_{\Gamma_q}$ for some $q\le n$. 
Such a representation was constructed in the paper \cite{EM}. 
Namely, for a positive real affine root $\tilde\alpha$ along the hyperplane 
$$
(\lambda,\tilde\alpha)+k(b-a)=0
$$
there is a representation $U$ of $H_{c,k}(\Gamma_q)$
which, as a $\Gamma_q$-module, looks like 
$\pi_{a,b}\otimes Y^{\otimes q}$, where 
$Y$ is an irreducible representation of $\Gamma$, and 
$\pi_{a,b}$ is the irreducible representation of $S_q$ 
whose Young diagram is a rectangle of width $a$ and height 
$b$ (so $q=ab$). The space $e_{\Gamma_q}U$ of $\Gamma_q$-invariants in $U$ is 
$e_{\Gamma_q}U=(\pi_{a,b}\otimes (Y^\Gamma)^{\otimes q})^{S_q}$. 

Let $\tilde\alpha=\alpha+N\delta$. Then $N=\dim Y^\Gamma$.
It is well known that $e_{\Gamma_q}U=0$ if and only if $N\le b-1$. 
Also, the hyperplane equation now looks like 
$$
(\lambda,\alpha)+k(b-a)+N=0.
$$

The number $m=b-a$ can take values between $1-n$ and $n-1$, i.e. $|m|\le n-1$. 
However, we have a restriction that the area of the rectangle is $q$: 
$1\le ab=b(b-m)=q\le n$. 

The larger root of the equation $b(b-m)=n$ with respect to $b$ is
$$
b_+= \sqrt{n+\frac{m^2}{4}}+\frac{m}{2}, 
$$
and the smaller root of this equation is negative. 
Therefore, if 
$$
0\le N\le b_+-1,
$$
then one can take $b=[b_+]$, and we have 
$1\le b\le b_+$, and also $b-m\ge 1$ (as $n\ge |m|+1$), so we obtain an aspherical representation. 
But this is exactly the condition in the theorem. 

Finally, we should separately consider the case $N=0$ and $\alpha<0$, 
i.e. $(\lambda,\alpha)+km=0$ (as in this case $\tilde\alpha$ is not a positive root). 
But in this case we can replace $\alpha$ with $-\alpha$ and $m$ with $-m$.

To pass from the generic point on the hyperplane to an arbitrary
point, one can use Proposition \ref{clo}. 

The theorem is proved. 
\end{proof}

\subsection{Aspherical locus and homological dimension}

One of the reasons aspherical values are interesting is the following theorem, 
connecting them to homological dimension. 

Let $V$ be a finite dimensional symplectic vector space, 
and $G\subset Sp(V)$ a finite subgroup.
Let $t\in \Bbb C$, $c$ be a conjugation-invariant $\Bbb C$-valued 
function on the set of symplectic reflections in $G$, and let 
$H=H_{t,c}(G,V)$ be the symplectic reflection algebra attached to $G,V,t,c$ (\cite{EG}). 
Let $e=e_G$ be the symmetrizing idempotent of $G$,
and $eHe$ be the corresponding spherical subalgebra. 

The proof of the following theorem was explained to me by R. Bezrukavnikov. 

\begin{theorem}\label{homde}
The algebra $eHe$ has finite homological dimension if and only if 
$HeH=H$ (in which case $eHe$ is Morita 
equivalent to $H$, so its homological dimension is $\dim V$).  
\end{theorem}

In other words, $(t,c)$ belongs to the aspherical locus if and only if $eHe$ has 
infinite homological dimension. 

The proof of Theorem \ref{homde} is based on the following general proposition.

\begin{proposition}\label{gene}
Let $A$ be an associative algebra with an idempotent $e$.
Assume that:

(i) $A$ has finite (left) homological dimension;

(ii) $eAe$ is a Gorenstein algebra 
(i.e., its dualizing module 
is an invertible bimodule), and $eA$ is a 
Cohen-Macaulay module over $eAe$ of dimension 0; 
and 

(iii) the natural map $\phi: A\to {\rm End}_{eAe}(eA)$ is an isomorphism. 

Then $eAe$ has finite homological dimension if and only if 
$AeA=A$. 
\end{proposition}

\begin{proof}
If $AeA=A$ then the functor $N\mapsto eN$ is an equivalence 
of categories $A-{\rm mod}\to eAe-{\rm mod}$
(i.e., $A$ and $eAe$ are Morita equivalent),
so the homological dimension of $eAe$ equals the homological 
dimension of $A$, i.e., is finite (so for this part of the proposition 
we only need condition (i)). It remains to show that if the 
homological dimension of $eAe$ is 
finite then $AeA=A$. 

Assume that $eAe$ has finite homological dimension. 
Consider the functor $F: eAe-{\rm mod}\to A-{\rm mod}$ 
defined by 
$$
F(N)={\rm Hom}_{eAe}(eA,N).
$$
We will show that conditions (i) and (ii) imply that 
the functor $F$ is exact, i.e., $eA$ is a projective $eAe$-module. 
Taking into account condition (iii), this means that 
$eA$ defines a Morita equivalence between $eAe$
and $A$, so $AeA=A$, as desired. 

To show that $F$ is exact, 
denote by $\omega_{eAe}$ the dualizing module 
for $eAe$, and note that by condition (ii), 
$$
{\rm Ext}^i_{eAe}(eA,\omega_{eAe})=0,\ i>0,
$$
which implies that for any projective $eAe$-module $P$, we have 
$$
{\rm Ext}^i_{eAe}(eA,P)=0,\ i>0
$$
(as $\omega_{eAe}$ is an invertible bimodule, and hence 
$P$ is a direct summand in $Y\otimes\omega_{eAe}$, where $Y$ is
a vector space). 
Now by induction in $m$, from the long exact sequence of cohomology 
it follows that if $N$ has a projective resolution of length $m$, 
then 
$$
{\rm Ext}^i_{eAe}(eA,N)=0,\ i>0.
$$
But the assumption that $eAe$ has finite homological dimension
implies that {\it any} $eAe$-module $N$ has 
a finite projective resolution. Thus, the last equality holds 
for any $eAe$-module $N$, which implies that the functor $F$ is exact. 
The proposition is proved. 
\end{proof} 

\begin{proof} (of Theorem \ref{homde}) 
To deduce Theorem \ref{homde} from Proposition \ref{gene}, 
it suffices to note that conditions (i)-(iii) 
of Proposition \ref{gene} 
hold for symplectic reflection algebras
(see \cite{EG}, Theorem 1.5).  
\end{proof}

Note that if $t=0$ then $eHe$ is commutative, and 
Theorem \ref{homde} reduces to
the statement that ${\rm Spec}(eHe)$
is smooth if and only if $HeH=H$, which is proved in \cite{EG}.
Moreover, by localizing $H$, one obtains 
a stronger result (also from \cite{EG}), stating 
that the smooth and Azumaya loci for $H$ coincide.

More generally, we have the following corollary 
of Proposition \ref{gene}.

Let $A$ be an associative algebra with 
an idempotent $e$, such that $eAe$ is a commutative 
finitely generated Gorenstein algebra, 
and $eA$ is a finitely generated 
Cohen-Macaulay $eAe$-module. 
Let $X={\rm Specm}(eAe)$, let $U_{\rm sm}\subset X$ 
be the smooth locus, and let $U_{\rm Az}\subset X$ 
be the Azumaya locus (namely, the set of points 
where $eA$ is locally free over $eAe$). 

\begin{corollary} If $A$ has finite homological dimension and 
the natural map $\phi: A\to {\rm End}_{eAe}(eA)$ is an 
isomorphism, then $U_{\rm sm}=U_{\rm Az}$. 
\end{corollary}

This corollary generalizes a result of Tikaradze \cite{T},
who proved, in particular, that $U_{\rm sm}=U_{\rm Az}$ for symplectic reflection algebras 
in positive characteristic.  

\section{Relation to the representation theory of affine Lie algebras}

\subsection{The setup}
The conjectural relation of the questions in Sections 3, 4 to representations of 
affine Lie algebras is based on the well known fact that 
the graded space $\oplus_{n\ge 0}K_0({\rm Rep} \Gamma_n)_{\Bbb C}$ 
has the same Hilbert series as ${\mathcal F}^{\otimes r+1}$, 
where ${\mathcal F}=\Bbb C[x_1,x_2,...]$ is the Fock space, and $r$ 
is the number of nontrivial conjugacy classes in $\Gamma$. 

Consider the Lie algebra $\widetilde{\g\oplus \Bbb C}$. It is spanned by elements 
$bt^j$, $b\in \g$, $j\in \Bbb Z$; $a_j$, $j\in \Bbb Z\setminus 0$; the scaling element $D$; and the central element $K$. 
The elements $a_i$ commute with $bt^j$, and we have 
$$
[D,a_i]=ia_i,\ [D,bt^i]=ibt^i,
$$
and
$$
[a_i,a_j]=i\delta_{i,-j}K,\ [at^i,bt^j]=[a,b]t^{i+j}+i\delta_{i,-j}(a,b)K,
$$ 
where $(,)$ is the invariant inner product on $\g$ normalized so that $(\alpha,\alpha)=2$ for all 
roots $\alpha$.

Let ${\bold V}_0$ be the basic representation of $\widetilde{\g}$ at level 1 (i.e., with highest weight $\omega_0$);
then ${\bold V}:={\bold V}_0\otimes {\mathcal F}$ is an irreducible representation of 
$\widetilde{\g\oplus \Bbb C}$, and by the Frenkel-Kac theorem (see \cite{K}),
the space ${\mathcal F}^{\otimes r+1}$ 
can be viewed as the sum of the weight subspaces of weight $\omega_0-n\delta$
in this representation; more precisely, one has 
$$
{\bold V}={\mathcal F}^{\otimes r+1}\otimes \Bbb C[Q_\g],
$$
where $Q_\g$ is the root lattice of $\g$, and for $\beta\in Q_\g$, 
${\mathcal F}^{\otimes r+1}\otimes e^\beta$ is the sum of weight subspaces of weight $\omega_0-
(n+\frac{\beta^2}{2})\delta+\beta$. 
So for either of the two filtrations defined in the previous sections 
there is a vector space isomorphism 
$$
{\rm gr}K_0({\rm Rep} \Gamma_n)_{\Bbb C}\cong {\bold V}[\omega_0-n\delta],
$$
(cf. \cite{FJW}) and the problem of finding the Poincare polynomials of the filtrations can be reformulated as the problem 
of describing the corresponding grading on ${\bold V}[\omega_0-n\delta]$. Below, we will state a conjecture of what 
this grading is expected to be. 

Fix $(\lambda,k)$. Define the Lie subalgebra $\ag=\ag(\lambda,k)$ of $\widetilde{\g\oplus \C}$ 
generated by the Cartan subalgebra $\widetilde{\h}$ and

1) $a_{ml}$, $l\in \Bbb Z\setminus 0$, if a singular hyperplane $E_{m,N}$ contains 
$(\lambda,k)$; 

2) $e_{\alpha+m\delta}, e_{-\alpha-m\delta}$ for each singular hyperplane $H_{\alpha,m,N}$ containing 
$(\lambda,k)$. 

Note that this Lie algebra inherits a polarization from $\widetilde{\g\oplus \Bbb C}$. 
Also note that it is finite dimensional if and only if $k\notin \Bbb Q$.

The above grading on ${\rm gr}K_0({\rm Rep}\Gamma_n)_{\Bbb C}$ can conjecturally be described in terms of the decomposition
of the representation ${\bold V}$ into $\ag$-isotypic components, in a way described below. 

We will say that $\mu\in \widetilde{\h}^*$ is a dominant integral weight for $\ag$ if 
$(\mu,\beta)\in \Bbb Z_+$ for any positive root $\beta$ of $\ag$.  
Denote the set of such weights by $P_+^\ag$. 
For $\mu\in P_+^\ag$, let $L_\mu$ be the corresponding irreducible 
integrable $\ag$-module. It is clear that 
$$
{\bold V}|_\ag\cong \oplus_{\mu\in P_+^\ag}L_\mu\otimes \Hom_\ag(L_\mu,\bold V).
$$

Let $\ag'\subset \ag$ be the subalgebra 
generated by the elements $e_{\alpha+m\delta}, e_{-\alpha-m\delta}$ for $(\lambda,k)\in H_{\alpha,m,N}$.
Note that it is a Kac-Moody algebra (finite dimensional or affine). 

\subsection{The main conjecture (the case of irrational $k$)}

Suppose first that $k$ is irrational. 
In this case, the Lie algebra $\ag$ is a finite
dimensional Levi subalgebra of $\widetilde{\g}$. 
Namely, $\ag$ is generated by $\widetilde{\h}$ and
$e_{\alpha+m_\alpha\delta}=e_\alpha t^{m_\alpha}$ 
for $\alpha$ running through some root subsystem 
$R'$ of the root system $R$ of $\g$. 

Our main conjecture in the case of irrational $k$ 
is the following. 

\begin{conjecture}\label{main1} For either of the filtrations $F_\bullet$
and ${\bold F}_\bullet$, there exists an isomorphism of vector spaces 
\begin{equation}\label{iso}
{\rm gr}_i K_0({\rm Rep}\Gamma_n)_{\Bbb C}\cong \oplus_{\mu\in P_+^\ag : \mu^2=-2i}
L_\mu[\omega_0-n\delta]\otimes \Hom_\ag(L_\mu,{\bold V}),
\end{equation}
\end{conjecture}

\begin{remark} \label{vani}
Note that if $\Hom_\ag(L_\mu,{\bold V})\ne 0$ then 
$\mu=\omega_0-j\delta+\beta$ where $\beta\in Q_\g$
and $\beta^2/2\le j$, so $\mu^2=\beta^2-2j$ is a nonpositive even integer.
Thus, the right hand side of (\ref{iso}) vanishes if $i<0$.
Also, if $L_\mu[\omega_0-n\delta]\ne 0$ then 
$\mu=\omega_0-n\delta+\sum_{i=0}^r p_i\alpha_i$, $p_i\ge 0$, so 
$\mu^2\ge -2n$. Thus, the right hand side of (\ref{iso}) 
vanishes if $i>n$. Therefore, Conjecture \ref{main1} is valid 
upon taking the direct sum over $i$. 
\end{remark}

In fact, the analysis of Remark \ref{vani} 
shows that in the extremal case $i=0$ (i.e., $\mu^2=0$), 
if $\phi: L_\mu\to {\bold V}$ is an $\ag$-homomorphism
and $v_\mu$ a highest weight vector of $L_\mu$, then 
$\phi(v_\mu)$ is necessarily an extremal 
vector of ${\bold V}$. The space of such vectors for each weight $\mu$ is
1-dimensional. Moreover, any such nonzero vector gives rise to a
nonzero homomorphism $\phi$. Indeed, for any simple root $\gamma$ of $\ag$ we have 
$(\mu+\gamma)^2=\mu^2+2(\mu,\gamma)+\gamma^2=2(\mu,\gamma)+2>0$, 
which implies that $\mu+\gamma$ is not a weight of $\bold V$, and hence
$\phi(v_\mu)$ is a highest weight vector for $\ag$. 

Thus, setting $i=0$, we 
obtain from Conjecture \ref{main1} the following conjecture on the 
number of finite dimensional representations. 

\begin{conjecture}\label{main2}
The number of isomorphism classes of 
finite dimensional irreducible representations of 
$H_{c,k}(\Gamma_n)$ is equal to 
$$
\sum_{\mu\in P_+^\ag: \mu^2=0}\dim L_\mu[\omega_0-n\delta]. 
$$
\end{conjecture}

Indeed, the above discussion implies that for $\mu=\omega_0-\frac{\beta^2}{2}\delta+\beta$ such that $\mu\in P_+^\ag$, 
one has $\dim \Hom_\ag(L_\mu,{\bold V})=1$. 

\begin{example}
Assume that $\ag'=({\mathfrak{sl}}_2)^\ell$. In this case, 
the point $(\lambda,k)$ lies on the intersection of the hyperplanes 
$(\lambda,\alpha^i)+km_i+N_i=0$, 
$i=1,...,\ell$, where $\alpha^i$ are pairwise orthogonal roots of $\g$, and 
$m_i$ are integers. 
Let $\widetilde{\alpha}^i=\alpha^i+m_i\delta$.
We will pick the signs in the hyperplane equations so that 
$\widetilde\alpha_i$ are positive roots; in particular, $m_i\ge 0$.    
Then, $\mu-\omega_0+n\delta$ should be a nonnegative 
integer linear combination of $\widetilde{\alpha}^i$:
$$
\mu-\omega_0+n\delta=\sum_{i=1}^\ell a_i\widetilde{\alpha}^i, a_i\in \Bbb Z_+.
$$ 
The condition $\mu^2=0$ reads 
\begin{equation}\label{diof}
\sum_i a_i(a_i+m_i)=n.
\end{equation}
Thus, Conjecture \ref{main2} predicts that the number of finite dimensional representations in 
this case equals the number of integer solutions $(a_1...,a_\ell)$ of equation (\ref{diof}) 
such that $a_i\ge 0$. 
 
If $k$ is a formal variable and $(\lambda,k)$ (where $\lambda=\lambda(k)$) is on the above hyperplanes
(but not on any other singular hyperplanes), then
this statement is in fact true, and can be deduced from the 
papers \cite{M} and \cite{G}.
Let us sketch a proof. 

The paper \cite{M} (and by another method, \cite{G}) 
constructs an irreducible representation 
of $H_{c,k}(\Gamma_n)$ for each solution of equation (\ref{diof})
as above. Indeed, fix such a solution $a=(a_1,...,a_\ell)$.  
Let $\pi_i$ be the representation of the symmetric group $S_{a_i(a_i+m_i)}$ 
whose Young diagram is a rectangle with $a_i$ columns and $a_i+m_i$ rows
if $\alpha^i+N_i\delta$ is a positive root, and with $a_i+m_i$ columns and $a_i$ rows otherwise. 
Let $Y_i$ be the representations of $\Gamma$ corresponding to dimension vectors 
$\pm(\alpha^i+N_i\delta)$ (whichever is a positive root), and let 
$$
U_a={\rm Ind}_{\prod_i S_{a_i(a_i+m_i)}}^{S_n}\bigotimes_i (\pi_i\otimes Y_i^{\otimes a_i(a_i+m_i)}).
$$
Then by \cite{M},\cite{G}, $U_a$ extends to an $H_{c,k}(\Gamma_n)$-module, and $U_a$ is not isomorphic to 
$U_{a'}$ (even as a $\Gamma_n$-module) if $a\ne a'$.  
Moreover, we claim that any finite dimensional irreducible $H_{c(k),k}(\Gamma_n)$-module is of this form. 
Indeed, suppose $U$ is a finite dimensional irreducible $H_{c(k),k}(\Gamma_n)$-module (defined over $\Bbb C[[k]]$), and 
$\bar U=U/kU$ be the corresponding module over $H_{c(0),0}(\Gamma_n)=\Bbb CS_n\ltimes H_{c(0)}(\Gamma)^{\otimes n}$. 
Since the roots $\alpha^i$ are orthogonal, the category of finite dimensional modules over the algebra 
$\Bbb CS_n\ltimes H_{c(0)}(\Gamma)^{\otimes n}$ is semisimple, so we have $\bar U=\oplus_{j=1}^J \bar U_j$, where $\bar U_j$ are simple modules. 
By \cite{M}, this means that each of the representations $\bar U_j$ should have a deformation 
$U_j$, and $U=\oplus_{j=1}^J U_j$ (as ${\rm Ext}^1(\bar U_i,\bar U_j)=0$), so $J=1$ and by \cite{M} $U=U_a$ 
for some $a$. 

Note that if the roots $\alpha^i$ are not orthogonal (i.e., $\ag'$ contains components 
of rank $>1$) then this proof fails (and the conjecture, in general, predicts more representations than 
constructed in \cite{M},\cite{G}). Presumably, the missing representations are 
constructed by $k$-deforming reducible but not decomposable representations of $H_{c,0}(\Gamma_n)$.    
\end{example}

\begin{example}
Suppose $n=1$. Let $\ag''=\ag'\cap \g$. 
Since the algebra $H_{c,k}(\Gamma_1)=H_c(\Gamma)$ does not depend on $k$,
it follows from the above that $\dim F_0=\dim {\bold F}_0$ is the semisimple 
rank of $\ag''$. Let us show that this agrees with the prediction of 
Conjecture \ref{main2} (so Conjectures \ref{main1} and \ref{main2} are true in this case). 
Indeed, if $L_\mu[\omega_0-\delta]\ne 0$, we have $\mu=\omega_0-\delta+\beta$, where $\beta$ is a root of 
$\ag''$. Also, for any positive root $\gamma$ of $\ag''$ we must have 
$(\mu,\gamma)=(\beta,\gamma)\ge 0$, which implies that $\beta$ is the maximal root of $\ag''$. 
So $\dim L_\mu[\omega_0-\delta]$ is the semisimple rank of $\ag''$, as desired. 
\end{example}

\subsection{The case of rational $k$}

\subsubsection{Integer $k$}
In the case when $k$ is rational, the situation becomes more complicated. 
First consider the situation when $k$ is an integer. 
In this case, the situation is similar to the case 
of irrational $k$, except that the Kac-Moody Lie algebra $\ag'$ is 
affine, rather than finite dimensional. More precisely, 
$\ag'=\widehat{\g'}$, where $\g'\subset \g$ is the semisimple subalgebra generated 
by the root elements corresponding to roots $\alpha$ with $(\lambda,\alpha)\in \Bbb Z$. 
Thus, Conjecture \ref{main1} predicts that 
$$
{\rm gr}_iK_0({\rm Rep} \Gamma_n)_{\Bbb C}\cong {\mathcal F}^{\otimes s}[i-n]\otimes {\mathcal F}^{\otimes r+1-s}[-i],
$$
where $s$ is the rank of $\g'$, and the numbers in the square brackets mean the degrees. 

\begin{theorem}\label{main3}
Conjectures \ref{main1} and \ref{main2} hold for $k=0$. 
\end{theorem}

\begin{proof}
This follows from the fact that $H_{c,0}(\Gamma_n)=\Bbb CS_n\ltimes H_c(\Gamma)^{\otimes n}$, so 
the understanding the filtrations $F_\bullet$ and ${\bold F}_\bullet$ reduces to the case $n=1$. 
\end{proof}

We expect that the situation is the same for any $k\in \Bbb Z$, because of 
the existence of translation functor $k\to k+1$ defined in \cite{EGGO}.
This functor is an equivalence of categories outside of aspherical hyperplanes,
and expected to always be a derived equivalence. 

\subsubsection{Non-integer $k\in \Bbb Q$}

Now consider the case of rational non-integer $k$, with denominator $m>1$. 
Then we have $\ag=\widetilde{\h}+\ag'+{\mathcal H}_m$, where 
${\mathcal H}_m$ is the Heisenberg algebra generated by $a_{mi}$.  
In this case, the one-index filtrations and gradings we have considered above can 
actually be refined to two-index ones. More specifically, 
it is easy to see that the possible annihilator varieties 
are the varieties $Y_{p,j}\subset V_n$ of points 
with some $p$ ($V$-valued) coordinates equal to zero and 
$j$ collections of $m$ other coordinates equal to each other
modulo $\Gamma$ (where $p,j\in \Bbb Z_+$ and $p+jm\le n$).\footnote{For example, 
for $m=2$ and $n=3$, $Y_{1,1}\subset V_3$ is the set of points 
$(v,v,0)$, $(v,0,v)$, and $(0,v,v)$.} Clearly, 
$Y_{p,j}$ contains $Y_{p+1,j}$, $Y_{p,j+1}$ and $Y_{p+m,j-1}$
and moreover $Y_{p,j}^\circ:=Y_{p,j}\setminus (Y_{p+1,j}\cup
Y_{p,j+1}\cup Y_{p+m,j-1})$ are locally closed smooth subvarieties 
which define a finite stratification of $V$. 
So we can define $F_{i,j}K_0({\rm Rep} \Gamma_n)$ to be spanned 
by the classes of the modules $M$ with $AV(M)\subset Y_{n-i-jm,j}$, 
and set 
$$
{\rm gr}_{i,j}^FK_0=F_{i,j}K_0/(F_{i-m,j+1}K_0+F_{i-1,j}K_0+F_{i,j-1}K_0).
$$
A similar definition is made for the filtration ${\bold F}$. 
Namely, we define $X_{p,j}\subset \Bbb C^n$ 
to be the varieties of points 
with some $p$ coordinates equal to zero and 
$j$ collections of $m$ of coordinates equal to each other modulo $\Gamma$, and define
$F_{i,j}K_0({\rm Rep} \Gamma_n)$ to be spanned by classes of modules 
with support (as a $\Bbb C[x_1,...,x_n]$-module) contained in $X_{n-i-jm,j}$. 
Then we set 
$$
{\rm gr}_{i,j}^{\bold F}K_0={\bold F}_{i,j}K_0/({\bold F}_{i-m,j+1}K_0+{\bold F}_{i-1,j}K_0+{\bold F}_{i,j-1}K_0).
$$
Note that for either of the filtrations, 
$$
{\rm gr}_sK_0=\oplus_{i,j\ge 0: i+j=s}{\rm gr}_{i,j}K_0
$$

Now we will formulate a conjectural description of the space ${\rm gr}_{i,j}K_0$. 
To this end, define the element 
$$
\partial_m=\frac{1}{m}\sum_{l=1}^\infty a_{-ml}a_{ml}.
$$
Obviously, $\partial_m$ has zero weight, and its eigenvalues are in $\Bbb Z_+$ on every module $L_\mu$. 
We denote by $L_\mu[\lambda,j]$ the subspace of $\bold V$ of weight $\lambda$ and 
eigenvalue $j$ of $\partial_m$. 

\begin{conjecture}\label{main4} For either of the filtrations $F_\bullet$
and ${\bold F}_\bullet$, there exists an isomorphism of vector spaces 
\begin{equation}\label{iso1}
{\rm gr}_{i,j} K_0({\rm Rep}\Gamma_n)_{\Bbb C}\cong \oplus_{\mu\in P_+^\ag : \mu^2=-2i}
L_\mu[\omega_0-n\delta,j]\otimes \Hom_\ag(L_\mu,{\bold V}).
\end{equation}
\end{conjecture}

Note that this subsumes Conjecture \ref{main1} (by putting $m=\infty$). 

For finite dimensional representations, $j$ must be zero, so 
we have 

\begin{conjecture}\label{main5}
The number of isomorphism classes 
of finite dimensional irreducible representations of 
$H_{c,k}(\Gamma_n)$ is equal to 
$$
\sum_{\mu\in P_+^\ag: \mu^2=0}\dim L_\mu[\omega_0-n\delta,0]. 
$$
\end{conjecture}

\begin{example}
Let $\Gamma=1$. In this case, $\ag={\mathcal H}_m$, 
and the Fock space ${\bold V}={\mathcal F}$ factors as 
${\mathcal F}={\mathcal F}_m\otimes {\mathcal F}_m'$, where 
${\mathcal F}_m$ is the Fock space for ${\mathcal H}_m$, 
and ${\mathcal F}_m'$ is the Fock space of the Heisenberg algebra ${\mathcal H}_m'$ 
generated by $a_s$ with $s$ not divisible by $m$. 
Also we have $\mu=\omega_0-i\delta$ for some $i\in \Bbb Z_+$. Thus Conjecture \ref{main4} predicts that 
$$
{\rm gr}_{i,j}K_0({\rm Rep}\Gamma_n)_{\Bbb C}=L_{\omega_0-i\delta}[\omega_0-n\delta,j]\otimes 
\Hom_{{\mathcal H}_m}(L_{\omega_0-i\delta},{\mathcal F}).
$$
For the first factor to be nonzero, we need $i=n-jm$, so the prediction
is that ${\rm gr}_{i,j}$ vanishes unless $i=n-jm$, in which case it simplifies to 
$$
{\rm gr}_{n-jm,j}K_0({\rm Rep}\Gamma_n)_{\Bbb C}=
{\rm gr}_{n-j(m-1)}K_0({\rm Rep}\Gamma_n)_{\Bbb C}=
{\mathcal F}_m[-jm]\otimes {\mathcal F}_m'[jm-n].
$$
In particular, there are no finite dimensional representations, since one never has $i,j=0$ (as $i+jm=n$). 

For filtration ${\bold F}$, 
this prediction is in fact a theorem.
Namely, first of all, it is clear 
from \cite{BE} that $p=n-i-jm$ must be zero, since 
the Weyl algebra has 
no finite dimensional representations.
Also, one has the following much stronger 
result, proved recently by S. Wilcox. 

Assume that $k<0$ (the case $k>0$ is similar). 
Let ${\mathcal C}_j$ be the category of modules 
over the rational Cherednik algebra $H_{1,k}(S_n)$
from category ${\mathcal O}$ which are supported on
$X_j:=X_{0,j}$ modulo those supported on $X_{j+1}$ (for $0\le j\le n/m$). 
Let ${\bold H}_q(N)$ be the finite dimensional Hecke algebra 
corresponding to the symmetric group $S_N$ and parameter $q$. 

\begin{theorem} (S. Wilcox, \cite{W})
The category ${\mathcal C}_j$ 
is equivalent to \linebreak
${\rm Rep}S_j\boxtimes {\rm Rep}{\bold H}_q(n-jm)$,
where $q=e^{2\pi ik}$.  

In more detail, the irreducible module $L_k(\lambda)$ 
whose lowest weight is the partition $\lambda$ of $n$ 
is supported on $X_j$ but not $X_{j+1}$ 
if and only if $\lambda=m\mu+\nu$, where 
$\mu$ is a partition of $j$ and $\nu$ 
is an $m$-regular partition (i.e. each part 
occurs $<m$ times), and in 
${\rm Rep}S_j\boxtimes {\rm Rep}{\bold H}_q(n-jm)$ 
it is expected to correspond to $\pi_\mu\boxtimes D_\nu$, where 
$\pi_\mu$ is the irreducible representation 
of $S_j$ corresponding to $\mu$, and $D_\nu$ is the irreducible 
Dipper-James module corresponding to $\nu$, \cite{DJ}. 
\end{theorem}
\end{example}

\begin{remark}
We note that in the important special case 
when $\ag'$ is the winding subalgebra 
$$\ag'=\g[t^m,t^{-m}]\oplus\Bbb C K,$$ 
$2\le m\le \infty$ (where for $m=\infty$, 
$\g[t^m,t^{-m}]\oplus \Bbb C K$ should be replaced by $\g$), 
the decomposition of 
$\bold V_0$ as an $\ag'$-module is completely described in 
the paper \cite{F2} (see also \cite{KW}). This decomposition easily 
yields the decomposition of $\bold V$ as an $\ag$-module, 
which appears in our conjectures. 
\end{remark}

\begin{remark}
After this paper appeared online, Conjecture \ref{main4} for cyclic groups $\Gamma$ and 
filtration ${\bold F}$ under some technical assumptions on the parameters was 
proved by Shan and Vasserot, \cite{SV}.  
\end{remark}

\subsection{Motivation for Conjectures \ref{pd} and \ref{rootun}}\label{pdmot}

Now we are ready to explain the motivation behind Conjectures 
\ref{pd} and \ref{rootun}. Namely, we expect that the
inner product $(,)$ on finite dimensional modules is obtained by restriction 
of the Shapovalov form on ${\bold V}$ to appropriate isotypic components.
Note that this was shown in \cite{SV} for cyclic $\Gamma$ and rational $k$, 
under some restriction on the other parameters (\cite{SV}, Remark 5.14). 
This would imply Conjecture \ref{pd}, since ${\bold V}$ is a unitary representation, and 
the Shapovalov form is positive definite. 
We also expect that the form $(,)_q$ is obtained by a 
similar restriction to isotypic components of the Shapovalov form for the quantum analog of ${\bold V}$,
(i.e. the basic representation ${\bold V}_q$ of the corresponding quantum affine algebra $U_q(\widetilde{\g\oplus \Bbb C})$).  
This would imply Conjecture \ref{rootun}, since this form degenerates only at roots of unity. 

\subsection{Relation to quantum connections}

Consider the special case when $(\lambda,k)$ is a generic point of a singular hyperplane
$H$. In this case, 
it was conjectured by Bezrukavnikov and Okounkov that the spaces ${\rm gr}_iK_0({\rm Rep}\Gamma_n)_{\Bbb C}$ 
should be eigenspaces of the residue $C_H$ on the hyperplane $H$ of the quantum connection on the equivariant quantum cohomology 
of the Hilbert scheme of the minimal resolution of the Kleinian singularity $\Bbb C^2/\Gamma$ (cf. \cite{MO}). 
This conjecture agrees with our conjectures (which, in the case of a generic point on a hyperplane, should 
be provable by using deformation theory). Namely, in the case of $\Gamma=1$ and the hyperplane $H=E_{m,N}$ 
(i.e., $k=-N/m$), this is explained in \cite{BE}; the case of general $\Gamma$ is similar, see 
\cite{MO}. So let us consider the remaining case 
$H=H_{\alpha,m,N}$. In this case, $\ag'$ is the ${\mathfrak{sl}}_2$-subalgebra 
corresponding to the root $\alpha+m\delta$, and $C_H=C_{\alpha+m\delta}$ 
is the Casimir of this ${\mathfrak{sl}}_2$-subalgebra. So the conjecture 
of Bezrukavnikov and Okounkov says that the decomposition into ${\rm gr}_i$ is the decomposition into eigenspaces 
of the Casimir $C_{\alpha+m\delta}$. But this decomposition is clearly the same as the decomposition 
into isotypic components for $\ag$, which implies that the conjectures agree. 

More generally, suppose $(\lambda,k)$ is a generic point of the intersection
of several singular hyperplanes. According to the insight of Bezrukavnikov and Okounkov, 
local monodromy of the quantum connection near $(\lambda,k)$ should preserve the space ${\rm gr}_i$ 
(after an appropriate conjugation). One also expects that in the neighborhood of $(\lambda,k)$ 
this connection is equivalent to its ``purely singular part'', i.e. the connection obtained from the 
full quantum connection by deleting all the regular terms of the Taylor expansion.  
This agrees with our conjectures, since the residues of the quantum connection 
on all the singular hyperplanes passing through $(\lambda,k)$ lie in (a completion of) $U(\ag)$
and hence preserve the decomposition of $\bold V$ into $\ag$-isotypic components.

\subsection{Relation to the work of Gordon-Martino and Shan
and to level-rank duality}

Let $\Gamma=\Bbb Z_\ell$ ($\ell=r+1$). We will assume for simplicity that 
$(\lambda,\alpha_i)=0$ for $i=1,...,\ell-1$ (we expect that the arguments below 
extend to the general case). 

\subsubsection{Irrational $k$}
First consider the case when $k$ is irrational. 
In this case, a construction 
due to I. Gordon and M. Martino and independently P. Shan \cite{Sh}
gives (under some technical conditions) a (projective) action 
of ${\mathfrak{gl}}_\infty$ of level $\ell$ on 
$\oplus_n K_0({\rm Rep}\Gamma_n)_{\Bbb C}=
{\mathcal F}^{\otimes \ell}$ of categorical origin. 
Here by ${\mathfrak{gl}}_\infty$
we mean the Lie algebra of infinite matrices with finitely many
nonzero diagonals \cite{K}, and the above 
representation is the $\ell$-th power of the 
usual Fock module (with highest weight $\omega_0$). 

The action of Gordon-Martino and Shan is based on induction and
restriction functors for the rational Cherednik algebra 
$H_{c,k}(\Gamma_n)$ constructed in \cite{BE}; in particular, 
finite dimensional representations are singular vectors (and one
should expect that conversely, any singular vector is a linear combination 
of finite dimensional representations). 

So let us find the singular vectors of this action of 
${\mathfrak{gl}}_\infty$. To this end, we need to decompose 
${\mathcal F}^{\otimes\ell}$ into isotypic components 
for ${\mathfrak{gl}}_\infty$. For this purpose, consider the product 
$$
E={\mathcal F}^{\otimes\ell}\otimes \Bbb C[\Bbb Z^\ell], 
$$
where for $a\in \Bbb Z^\ell$ the vector $1\otimes e^a$ is put in degree 
$-a^2/2$. By the boson-fermion correspondence (\cite{K}),  
$$
E=(\wedge^{\frac{\infty}{2}+\bullet} U)^{\otimes \ell},
$$ 
where $U$ is the vector representation of 
${\mathfrak{gl}}_\infty$, and $\wedge^{\frac{\infty}{2}+\bullet}$ 
denotes the space of semiinfinite wedges (of
all degrees). Thus, 
$$
E=\wedge^{\frac{\infty}{2}+\bullet} 
(U\otimes \Bbb C^\ell).
$$
By an infinite dimensional analog of the Schur-Weyl duality
between ${\mathfrak{gl}}(U)$ and ${\mathfrak{gl}}(W)$ 
inside $\wedge (U\otimes W)$ for finite dimensional spaces $U$
and $W$ (which is a limiting case of level-rank duality, see below), 
this implies that the centralizer of ${\mathfrak{gl}}_\infty$ 
in this representation is the algebra generated by 
${\mathfrak{gl}}_\ell$ acting by degree-preserving transformations. 
More precisely, we have a decomposition (\cite{F1}, Theorem 1.6):
$$
E=\oplus_\nu {\mathcal F}_{\nu^*}\otimes L_\nu,
$$
where $\nu$ runs over partitions with at most $\ell$ parts, 
$L_\nu$ are the corresponding representations 
of ${\mathfrak{gl}}_\ell$, and ${\mathcal F}_{\nu^*}$ 
is the irreducible representation of ${\mathfrak{gl}}_\infty$
corresponding to the dual partition $\nu^*$. 
Thus, the space of singular vectors is given by the formula
$$
E_{\rm sing}=\oplus_\nu L_\nu,
$$
and hence 
$$
{\mathcal F}^{\otimes \ell}_{\rm sing}=\oplus_\nu L_\nu[0]. 
$$
Moreover, $L_\nu[0]$ sits in degree $-\nu^2/2$. 
Thus we get that the number of finite dimensional representations 
equals 
$$
\sum_{\nu: \nu^2=2n}\dim L_\nu[0],
$$
as predicted by Conjecture \ref{main2}.

\subsubsection{Rational $k$} 

Now assume that $k$ is a rational number with denominator $m>1$.
Then the construction of Gordon-Martino and Shan \cite{Sh} yields an action 
of $\widehat{\mathfrak{sl}}_m$ of level $\ell$ on 
$\oplus_n K_0({\rm Rep}\Gamma_n)_{\Bbb C}=
{\mathcal F}^{\otimes \ell}$, 
where ${\mathcal F}$ is the Fock representation 
of ${\mathfrak{gl}}_\infty$ of highest weight $\omega_0$
restricted to $\widehat{\mathfrak{sl}}_m\subset 
{\mathfrak{gl}}_\infty$. Namely, the action is defined by induction and 
restriction functors of \cite{BE}. We expect that by using the restriction functors 
${\mathcal O}_{c,k}(\Gamma_n)\to {\mathcal O}_{c,k}(\Gamma_{n-mj})$ 
(restricting to the locus with $j$ $m$-tuples of equal coordinates),
and the induction functors in the opposite direction, this action 
can be upgraded to an action of $\widehat{\mathfrak{gl}}_m$. 
\footnote{I have been told by I. Losev that this step is quite
nontrivial and requires passing to derived functors.}
 
As before, finite dimensional representations are singular vectors for this action, 
and one should expect that conversely, any singular vector is a linear combination 
of finite dimensional representations. 

So let us find the singular vectors of this action of 
$\widehat{\mathfrak{gl}}_m$. To this end,
we need to decompose ${\mathcal F}^{\otimes \ell}$
into isotypic components for $\widehat{\mathfrak{gl}}_m$.
For this purpose, as before, consider the space $E$. 
The space $U$ considered above now equals 
$\Bbb C^m[t,t^{-1}]$, so the boson-fermion correspondence yields
$$
E=\wedge^{\frac{\infty}{2}+\bullet} 
(\Bbb C^m\otimes \Bbb C[t,t^{-1}])^{\otimes \ell},
$$ 
Thus, 
$$
E=\wedge^{\frac{\infty}{2}+\bullet} 
(\Bbb C^m\otimes \Bbb C[t,t^{-1}]\otimes \Bbb C^\ell).
$$
By an affine analog of the Schur-Weyl duality,
this implies that the centralizer of $\widehat{\mathfrak{gl}}_m$ 
in this representation is the algebra generated by 
$\widehat{\mathfrak{sl}}_\ell$ acting at level $m$. 
 
More precisely, this is an instance of the 
phenomenon called {\it the level-rank duality} (cf. \cite{F1}).
Namely, let $P_{+0}^{m,\ell}$ 
be the set of dominant integral weights 
for $\widehat{\mathfrak{sl}}_m$ at level $\ell$ 
which are trivial on the center of $SL_m$. 
(We don't take into account the action of $D$ here, so this is a finite set). 
Then level-rank duality implies that there is a natural bijection 
$\dagger: P_{+0}^{m,\ell}\cong P_{+0}^{\ell,m}$, and 
we have a decomposition 
$$
E=\oplus_{\nu\in P_{+0}^{\ell,m}} L_{\nu^\dagger}\otimes {\mathcal F}\otimes L_{\nu},
$$
where $\widehat{\mathfrak{gl}}_m$ acts irreducibly in the product of the 
first and the second components, and $\widehat{\mathfrak{sl}}_\ell$ acts irreducibly on the third 
component in each summand. 

This shows that the space of singular vectors is 
$$
E_{\rm sing}=\oplus_{\nu\in P_{+0}^{\ell,m}} L_\nu. 
$$
Now, the highest weight vector of $L_\nu$ sits in degree $-\nu^2/2$, so 
we get that the number of finite dimensional irreducible representations
of $H_{c,k}(\Gamma_n)$ is 
$$
\sum_{\nu\in P_{+0}^{\ell,m}: n-\frac{\nu^2}{2}\in m\Bbb Z_{\ge 0}} \dim L_\nu[0,\frac{\nu^2}{2}-n],
$$
where the first entry in the square brackets is weight under $\h$, 
and the second one is the degree (noting that the degrees of vectors 
in $L_\nu$ are multiples of $m$). This is exactly the prediction of 
Conjecture \ref{main5}.

\end{document}